\documentclass[11pt,reqno]{amsart}

\usepackage[T1]{fontenc}
\usepackage[english]{babel}

\usepackage[margin=1.2in,footskip=.2in]{geometry}\usepackage{amsmath,amssymb,amscd,amsfonts,mathtools,tikz,caption,float}
\usetikzlibrary{patterns}
\usepackage[linkcolor=blue,colorlinks=true,citecolor=blue]{hyperref}
\usepackage[capitalise]{cleveref}
\newtheorem{Thm}{Theorem}

\newtheorem{thm}{Theorem}[section]
\newtheorem{lm}[thm]{Lemma}

\newtheorem{coro}[thm]{Corollary}

\newcommand{\abs}[1]{\left\vert#1\right\vert}

\newcommand{\R}{\mathbb{R}}

\newcommand{\h}{\mathbb{H}}
\newcommand{\Q}{\mathbb{Q}}
\newcommand{\Z}{\mathbb{Z}}

\newcommand{\bbm}{\begin{bmatrix}}
\newcommand{\ebm}{\end{bmatrix}}
\newcommand{\bpm}{\begin{pmatrix}}
\newcommand{\epm}{\end{pmatrix}}
\newcommand{\bsm}{\left(\begin{smallmatrix}}
\newcommand{\esm}{\end{smallmatrix}\right)}
\newcommand{\bsbm}{\left[\begin{smallmatrix}}
\newcommand{\esbm}{\end{smallmatrix}\right]}

\newcommand{\sign}{\mathrm{sign}}
\newcommand{\vol}{\mathrm{vol}}
\newcommand{\tr}{\mathrm{tr}}
\newcommand{\1}{\mathrm{1}}
\newcommand{\SL}{\mathrm{SL}}
\newcommand{\PSL}{\mathrm{PSL}}

\begin{document}

\title[Reciprocity of Dedekind sums and the Euler class]{Reciprocity of Dedekind sums and the Euler class}
\author{Claire Burrin}
\address{Dept. of Mathematics, Rutgers University, 110 Frelinghuysen Rd, Piscataway, NJ 08854}
\email{claire.burrin@rutgers.edu}
\date{\today}
\maketitle

\begin{abstract}
Dedekind sums are arithmetic sums that were first introduced by Dedekind in the context of elliptic functions and modular forms, and later recognized to be surprisingly ubiquitous. Among the variations and generalizations introduced since, there is a construction of Dedekind sums for lattices in $\SL_2(\R)$. Building upon work of Asai, we prove the reciprocity law for these Dedekind sums, based on a concrete realization of the Euler class. As an application, we obtain an explicit formula for Dedekind sums on Hecke triangle groups in terms of continued fractions.\\ 
%
\end{abstract}

\section*{Introduction}

Let $(\!(\cdot)\!):\R\to\left(-\frac12,\frac12\right)$ be the sawtooth function defined by
\begin{align*}
(\!(x)\!) = \begin{dcases} x - \lfloor x\rfloor - 1/2 & x\not\in\Z, \\ 0 & x\in \Z, \end{dcases}
\end{align*}
where $\lfloor x\rfloor$ is the largest integer $\leq x$.
Dedekind \cite{Ded} introduced the arithmetic sums
\begin{align}\label{DS}
s(d,c) = \sum_{k=1}^{c-1} \left(\!\left(\frac{k}{c}\right)\!\right)\left(\!\left(\frac{kd}{c}\right)\!\right),
\end{align}
for $d,c$ coprime integers, in connection to the modular transformation of $\log\eta$, the logarithm of the Dedekind $\eta$-function,
$$
\eta(z) = e^{\pi iz/12}\prod_{n\geq1} \left(1-e^{2\pi inz}\right).
$$
From that transformation, Dedekind further deduced the beautiful identity
\begin{align}\label{RL}
s(d,c) + s(c,d) = \frac{1}{12}\left(\frac{d}{c} + \frac{1}{dc} +\frac{c}{d}\right) - \frac{1}{4}.
\end{align}
Beyond being esthetically appealing, this reciprocity law has an immediate practical purpose; together with the more obvious observations
\begin{align*}
s(0,1) &=0 \quad \text{and}  \\  s(d',c) &= s(d,c)\quad \text{if}\quad d'\equiv d\mod c,
\end{align*}
it allows for the fast computation of values of Dedekind sums via the Euclidean algorithm. (In terms of applications, this is all the more significant when considering, say, the various interactions between Dedekind-type sums and pseudo-random number generation.)

Dedekind's reciprocity law (\ref{RL}) can be seen as a special case of the more general formula deduced by Dieter \cite{Dieter} (once again from the transformation of $\log\eta$),
\begin{align}\label{RLR}
s(d_1,c_1) + s(d_2,c_2) + s(d_3,c_3) = \frac{1}{12}\left(\frac{c_1}{c_3 c_2} + \frac{c_2}{c_1 c_3} + \frac{c_3}{c_2 c_1}\right) -\frac{1}{4}
\end{align}
for $c_i$, $d_i$ given by the relation
\begin{align*}
\bpm a_1 & b_1 \\ c_1&d_1 \epm\bpm a_2 & b_2 \\ c_2&d_2 \epm\bpm a_3 & b_3 \\ c_3&d_3 \epm = \bpm 1&\\&1\epm,
\end{align*}
where each matrix is an element of $\SL_2(\Z)$.

Following an idea of Kubota, Asai \cite{Asai} argues that these reciprocity formulas are consequences of a deeper mechanism underlying the relation between Dedekind sums and the transformation of $\log\eta$. Indeed, he shows that Dieter's reciprocity law (\ref{RLR}) 
can be derived without relying explicitly on the theory of modular forms (which is what Dieter, and before him, Dedekind, had done) but instead from a careful investigation of the "splitting" of the central extension
\begin{align*}
0 \to \Z \to \widetilde{\SL_2(\R)} \to \SL_2(\R) \to 1
\end{align*}
over $\SL_2(\Z)$, where $\widetilde{SL_2(\R)}$ denotes the universal covering group of $\SL_2(\R)$. To explain what this means, we first reinterpret this statement in terms of the Euler class. It is a standard fact that isomorphism classes of central extensions are classified by cohomology in degree 2, and that the second cohomology group $\mathrm{H}^2(\SL_2(\R))$ is one-dimensional, and generated by the Euler class. In Asai's terminology, the "splitting" refers to the existence of a map $\rho:\SL_2(\Z)\to\R$ satisfying the relation
\begin{align}\label{TRIV}
\rho(\gamma \tau)-\rho(\gamma)-\rho(\tau) = \omega(\gamma,\tau),
\end{align}
where $\omega$ is a 2-cocycle representative of the Euler class. In these terms, Dieter's reciprocity law (\ref{RLR}) can be tracked down to two features of the chosen representative $\omega$
\begin{enumerate}
\item[(I)] the Dedekind sums are determined by a function $\rho:\SL_2(\Z)\to\R$ satisfying (\ref{TRIV}),
\item[(II)] there is an explicit formula expressing $\omega$.
\end{enumerate}
Consequently, Asai suggests that (\ref{TRIV}) be named the generalized reciprocity law. Using this framework, we will deduce Dieter's reciprocity law (\ref{RLR}) for generalized Dedekind sums arising from lattices $\Gamma<\SL_2(\R)$. 

We conclude that, while (\ref{DS}), (\ref{RL}) and (\ref{RLR}) are arithmetic formulas, and can as well be proven purely arithmetically (see \cite{Rad}), their underlying mechanism is not only independent of arithmetic -- that was already clear from Dedekind's original proof -- but also independent of any arithmetic properties of the group at large.

Moreover, the cohomological considerations above indicate that any $\rho:\Gamma\to\R$ satisfying (\ref{TRIV}) relates to the generalized Dedekind sums coming from $\Gamma$. Atiyah famously identified seven (roughly) equivalent definitions, coming from geometry, topology, physics, number theory, of $\rho$ for $\SL_2(\Z)$ \cite{Ati}, a list that was further complemented by Barge and Ghys, who made explicit the relation to the Euler class \cite{BG}. The identification of general invariants for hyperbolic surfaces that thus relate to Dedekind sums is an intriguing program, and the object of further research.

\subsection*{Presentation of results}

If $\Gamma<\SL_2(\R)$ is a lattice with cusp(s), there is a known construction
of maps $\rho:\Gamma\to\R$ satisfying (\ref{TRIV}) in the automorphic forms literature, with connections to the Kronecker first limit formula and to the theory of multiplier systems \cite{Go,Hej,Pat}. Using this construction, the author introduced in \cite{moi} Dedekind symbols, whose definition we now review.

Fix a cusp $\frak{a}$ of $\Gamma$. Let $\Gamma_\frak{a}$ be the isotropy subgroup of $\frak{a}$. Fix a scaling $\sigma_\frak{a}$. Each non-trivial double coset $[\![\gamma]\!]=\Gamma_\frak{a}\gamma\Gamma_\frak{a}$, $\gamma\in\Gamma$, yields the Dedekind symbol 
\begin{align*}
\mathcal{S}_\frak{a}\left([\![\gamma]\!]\right) = \frac{\vol(\Gamma\backslash\h)}{4\pi}\frac{a_\gamma+d_\gamma}{c_\gamma} - \rho_\frak{a}\bpm a_\gamma&b_\gamma\\ c_\gamma&d_\gamma\epm - \frac14 \sign(c_\gamma)
\end{align*}
where $\bsm a_\gamma&b_\gamma\\ c_\gamma&d_\gamma\esm = \sigma_\frak{a}^{-1}\gamma\sigma_\frak{a}$ and where $\rho_\frak{a}$ verifies (\ref{TRIV}). We note that the definition of Dedekind symbol does not actually depend on the particular choice of scaling $\sigma_\frak{a}$. 
While the algebraic structure of double cosets is very useful in practice, it is not the most intuitive. The following proposition shows that the Dedekind symbol $\mathcal{S}_\frak{a}$ can alternatively be seen as a function on the cusp set $\mathcal{O}=\Gamma.\frak{a}$ of $\Gamma$.

\begin{Thm}\label{THM0}
Let $\mathcal{O}=\Gamma.\frak{a}$ be the orbit of $\frak{a}$ under the action of $\Gamma$ by fractional linear transformations, endowed with the equivalence relation $[x]=[y]$ if there exists $\gamma\in\Gamma_\frak{a}$ such that $\gamma x=y$. Let $\mathcal{O}'=\Gamma.\frak{a}\setminus\{\frak{a}\}$. There is a one-to-one correspondance between equivalence classes $[x]$, $x\in\mathcal{O}'$, and non-trivial double cosets in $\Gamma_\frak{a}\backslash\Gamma\slash\Gamma_\frak{a}$. Consequently, we may define
$$
\mathcal{S}_\frak{a}([\gamma.\frak{a}]) = \mathcal{S}_\frak{a}([\![\gamma]\!]).
$$
\end{Thm}
In the case of $\SL_2(\Z)$, the latter result expresses the one-to-one correspondance between coprime integers and $\Q$, and allows to regard Dedekind sums as a function on the cusp set of the modular group, minus the point at infinity. 

Our main result expresses Dieter's reciprocity law (\ref{RLR}) for Dedekind symbols.

\begin{Thm}\label{THM1}
In the notation introduced above,
\begin{align*}
\mathcal{S}_\frak{a}([\![\gamma_1]\!]) + \mathcal{S}_\frak{a}([\![\gamma_2]\!]) + \mathcal{S}_\frak{a}([\![\gamma_3]\!]) = \frac{\vol(\Gamma\backslash\h)}{4\pi}\left(\frac{c_{\gamma_1}}{ c_{\gamma_3}c_{\gamma_2}} + \frac{c_{\gamma_2}}{c_{\gamma_1} c_{\gamma_3}} + \frac{c_{\gamma_3}}{ c_{\gamma_2} c_{\gamma_1}}\right) -\frac14\sign(c_{\gamma_1} c_{\gamma_2} c_{\gamma_3}),
\end{align*}
for all $\gamma_i\in\Gamma$ such that $\gamma_1 \gamma_2 \gamma_3=I$, and $[\![\gamma_i]\!]$ are non-trivial double cosets in $\Gamma_\frak{a}\backslash\Gamma\slash\Gamma_\frak{a}$.
\end{Thm}

On the other hand, the generalization of Dedekind's reciprocity law (\ref{RL}) only makes sense for groups that contain the involution $S=\bsm &-1\\1&\esm$. In the second part of this article, we restrict our attention to the Hecke triangle groups $G_q$. 

Recall that $G_q$ is a discrete triangle group of type $(2,q,\infty)$, where $q\geq 3$. This means that $G_q$ is generated by reflections on the sides of a triangle with interior angles $(\pi/2,\pi/q,0)$. Algebraically, $G_q$ is generated by the involution $S$ together with the translation $T_q=\bsm 1 & \lambda_q \\ &1\esm$, where $\lambda_q=2\cos(\pi/q)$. Since there is only one cusp, hence one Dedekind symbol, we will write $\mathcal{S}$ instead of $\mathcal{S}_\infty$, and consider $\mathcal{S}$ as a function on equivalence classes of 
$$
\mathcal{O}=\{\gamma.\infty \text{ mod }\lambda_q :\gamma\in G_q\},
$$
as per the correspondance stated in \cref{THM0}.

\begin{Thm}\label{thm HTG}
If $\frac{a}{c}\in\mathcal{O}$ with $ac\neq0$, then
\begin{align*}
\mathcal{S}\left(\left[\frac{a}{c}\right]\right) - \mathcal{S}\left(\left[\frac{c}{a}\right]\right)\ =\ \frac{1-\frac{2}{q}}{8\cos(\frac{\pi}{q})}\left(\frac{a}{c} +\frac{1}{a c} + \frac{c}{a}\right) - \frac{1}{4}\sign(ac)
\end{align*}
\end{Thm}

Furthermore, points in $\mathcal{O}$ can be expanded in $\lambda_q$-continued fractions. In fact, each element $\gamma\in G_q$ can be expressed as a word in the group generators $S_q= \bsm 1&\\ \lambda_q&1\esm$ and $T_q=\bsm 1&\lambda_q\\ &1\esm$, and 
\begin{align*}
\left(S_q^{a_1} T_q^{a_2} \cdots T^{a_{n-1}}_q S_q^{a_n}\right).\infty\ =\ \cfrac{1}{a_1\lambda_q + \cfrac{1}{a_2\lambda_q +\cfrac{1}{\dots}}}.
\end{align*}
We denote the RHS of this equation by $ [a_1,a_2,\dots,a_n]$. By applying the reciprocity law of \cref{thm HTG} recursively, we obtain an explicit formula for the Dedekind symbol $\mathcal{S}$ in terms of $\lambda_q$-continued fractions, which generalizes a theorem of Hickerson for the Dedekind sums \cite[Thm.~1]{Hickerson}.

\begin{Thm}\label{THM3}
Let $[a_1,a_2,\dots,a_n]$ be a finite $\lambda_q$-continued fraction expansion as above. Then
\begin{align*}
\mathcal{S}([a_1,a_2,\dots,a_n])& =\\
\frac{1-\frac{2}{q}}{4\lambda_q} &\left( [a_1,\dots,a_n] +(-1)^{n+1}[a_n,\dots,a_1] - \lambda_q\sum_{j=1}^n(-1)^j a_j\right)- \frac{1-(-1)^n}{8}.
\end{align*}
\end{Thm}

\section{Notation and terminology}
Let $\h$ denote the hyperbolic upper half-plane, and recall that $\SL_2(\R)$ acts on $\h$ by fractional linear transformations, and that this action factors through $\PSL_2(\R)$. Throughout the article, $\Gamma$ will denote a cofinite Fuchsian group. That is, a discrete subgroup of $\SL_2(\R)$ of finite covolume $V=\vol(\Gamma\backslash\h)<\infty$, containing at least one parabolic element. We recall that $\gamma\in\Gamma$ is parabolic if $\abs{\tr(\gamma)}=2$ or, equivalently, if the action of $\gamma$ on $\overline{\h}=\h\cup\partial\h$ fixes a single point and that this point is in $\partial\h =\R\cup\{\infty\}$. Such a point is referred to as a cusp and will be denoted $\frak{a}$. For each cusp $\frak{a}$, the isotropy subgroup of elements of $\Gamma$ fixing $\frak{a}$ is denoted by $\Gamma_\frak{a}$. It is, up to sign, an infinite cyclic subgroup of $\Gamma$. A matrix $\sigma_\frak{a}\in\SL(2,\R)$ is called a scaling if it verifies $\sigma_\frak{a}(\infty)=\frak{a}$ and
$$
\sigma_\frak{a}^{-1}\Gamma_\frak{a}\sigma_\frak{a} =\left(\sigma_\frak{a}^{-1}\Gamma\sigma_\frak{a}\right)_\infty = \pm \bpm 1&\Z\\&1\epm.
$$
These two conditions do not determine uniquely $\sigma_\frak{a}$ but up to right multiplication by any element of the group $\bsm 1&\R\\&1\esm$. There is a one-to-one correspondence between subgroups of $\PSL_2(\R)$ and subgroups of $\SL_2(\R)$ that contain $-I=\bsm -1&\\&-1\esm$. We shall always assume that $-I\in\Gamma$.

\section{A concrete realization of the Euler class}

\subsection{Petersson's cocycle}
We review the classical construction investigated by Petersson \cite{Pet0,Pet}, which provides an explicit bounded Euler class representative $\omega$. For any $z\in\h$, set
\begin{align}\label{omega}
\omega(g,h) = \frac{1}{2\pi i}\left( \log j(g,h z) + \log j(h, z) - \log j(gh,z)\right),
\end{align}
where $\log$ denotes the principal branch of the logarithm, that is $\log(cz+d) = \ln\abs{cz+d} +i \arg(cz+d)$ for $-\pi<\arg(cz+d)\leq\pi$, and where $j(g,z)=cz+d$, for $g=\bsm *&*\\ c&d\esm\in\SL_2(\R)$, is the usual automorphy factor, which satisfies
\begin{align}\label{cocycle}
j(gh, z) = j(g,h z)j(h,z).
\end{align}

\begin{lm} The function $\omega$ defines a bounded 2-cocycle representative of the Euler class.
\end{lm}

\begin{proof}
Observe that (\ref{cocycle}) implies that the LHS of (\ref{omega}) is real-valued. Since it is also holomorphic in $z$, it must be constant; $\omega$ is indeed independent of the choice of $z$. By definition, $\omega$ can only take integer values. Since moreover $\abs{\omega(g,h)}\leq \frac{3}{2}$, we conclude that $\omega$ only takes values in $\{-1,0,1\}$. Finally, one can check by direct computation that
$$
\omega(g_1, g_2) + \omega(g_1 g_2,g_3) = \omega(g_1,g_2 g_3) + \omega(g_2, g_3).
$$
Hence $[\omega]\in\mathrm{H}^2(\SL_2(\R))$ and we conclude from the previous observations that $[\omega]$ coincides with the Euler class.
\end{proof}

\subsection{Computing values of $\omega$}
The 2-cocycle (\ref{omega}) can be computed on explicit elements. Asai provides a clean formula to do so \cite[Thm.~2]{Asai}, simplifying the laborious presentation of Petersson \cite{Pet}. However, his formula is obtained under the unusual choice $-\pi\leq\arg(cz+d)<\pi$. By a simple reparametrization, we recover that formula for the principal branch of the logarithm. 
\begin{thm}\label{omega values}
Set
\begin{align*}
c(-d)= \begin{dcases} c & c\neq0,\\ -d & c=0,
\end{dcases}
\quad
\text{ and }
\quad
\sign(x) = \begin{dcases} 1 & x>0, \\ 0 & x=0, \\ -1 & x<0.\end{dcases}
\end{align*}
Then 
\begin{align}\label{explicit}
\omega(g,h)=\frac{1}{4}(\sign(c_g(-d_g))+ &\sign(c_h(-d_h)) \nonumber \\ &-\sign(c_{gh}(-d_{gh})) -\sign(c_g(-d_g) c_\tau(-d_h) c_{gh}(-d_{gh}))).
\end{align}
\end{thm}

\begin{coro}\label{corol}
The explicit values of $\omega(g,h)$ are given by the following table.
\begin{align}\label{values}
\begin{array}{ccc|c}
\sign\left(c_g(-d_g)\right) & \sign\left(c_h(-d_h)\right) & \sign\left(c_{gh}(-d_{gh})\right) & \omega(g,h)\\
\hline
1 & 1 & -1 & 1\\
-1 & -1 & 1 & -1\\
&\text{otherwise}&& 0
\end{array}
\end{align}
\end{coro}

\begin{proof}
One can check directly the validity of the identity
$$
\log(cz+d) = \log\left(\frac{cz+d}{i\ \sign(c(-d))}\right) +i\frac{\pi}{2}\sign\left(c(-d)\right).
$$
In showing by inspection that
$$
\log\left(\frac{j(g,h z)}{i\ \sign(c_g(-d_g))}\right) + \log\left(\frac{j(h, z)}{i\ \sign(c_h(-d_h))}\right) - \log\left(\frac{j(gh, z)}{i\ \sign(c_{gh}(-d_{g}))}\right)
$$
$$ = -i\frac{\pi}{2}\sign(c_g(-d_g) c_\tau(-d_h) c_{gh}(-d_{gh})),
$$
we obtain the formula (\ref{explicit}), from which it is easy to complete table (\ref{values}).
\end{proof}

\section{Dedekind symbols and reciprocity}

\subsection{Definition of Dedekind symbols}
Let $\Gamma$ be a cofinite Fuchsian group with a cusp $\frak{a}$ and fix a scaling $\sigma_\frak{a}$. The Dedekind symbol $\mathcal{S}_\frak{a}$ is defined on non-trivial double cosets of $\Gamma_\frak{a}\backslash\Gamma\slash\Gamma_\frak{a}$ by
\begin{align*}
\mathcal{S}_\frak{a}\left([\![\gamma]\!]\right) = \frac{\vol(\Gamma\backslash\h)}{4\pi}\frac{a_\gamma+d_\gamma}{c_\gamma} - \rho_\frak{a}\bpm a_\gamma&b_\gamma\\ c_\gamma&d_\gamma\epm - \frac14 \sign(c_\gamma)
\end{align*}
for $\bsm a_\gamma&b_\gamma\\ c_\gamma&d_\gamma\esm = \sigma_\frak{a}^{-1}\gamma\sigma_\frak{a}$, and $\rho_\frak{a}:\sigma_\frak{a}^{-1}\Gamma\sigma_\frak{a}\to\R$ satisfying
\begin{align}\label{AGR}
\rho_\frak{a}(\gamma\tau) - \rho_\frak{a}(\gamma) - \rho_\frak{a}(\tau) = \omega(\gamma,\tau),
\end{align}
where $\omega$ is the bounded 2-cocycle given by (\ref{omega}), and the definition does not depend on the particular scaling $\sigma_\frak{a}$ \cite[Thm.~2]{moi}. We will not review here the construction of $\rho_\frak{a}$ or the connection of Dedekind symbols to automorphic forms (in the form of generalized $\log\eta$-functions). Instead we will highlight the relevant features of this definition.

The most important characteristic of this definition is that the Dedekind symbols factor through (non-trivial) double cosets in $\Gamma_\frak{a}\backslash\Gamma\slash\Gamma_\frak{a}$. By conjugation with $\sigma_\frak{a}$, a double coset representative $\gamma$ of $[\![\gamma]\!]$ gets sent to
\begin{align}\label{dcd}
\left(\sigma_\frak{a}^{-1}\Gamma_\frak{a}\sigma_\frak{a}\right)\bpm a_\gamma&b_\gamma\\ c_\gamma&d_\gamma \epm \left(\sigma_\frak{a}^{-1}\Gamma_\frak{a}\sigma_\frak{a}\right) = \pm\bpm a_\gamma+c_\gamma\Z&*\\ c_\gamma&d_\gamma+c_\gamma\Z\epm. 
\end{align}
In consequence, we observe that $[\![\gamma]\!]$ is non-trivial if and only if $c_\gamma\neq0$.  Moreover, one can always choose a representative $\gamma$ of $[\![\gamma]\!]$ such that $c_\gamma>0$, and $0\leq a_\gamma,d_\gamma<c_\gamma$. These two simple observations are primordial in establishing the Dedekind symbol as the natural generalization of the Dedekind sums. (Of course, the two definitions coincide on $\SL_2(\Z)$.)

\subsection{Equivalent definition}
We show that alternatively, and equivalently, the Dedekind symbol $\mathcal{S}_\frak{a}$ can be defined as a periodic function on the orbit (or cusp set) $\mathcal{O}'=\Gamma.\frak{a}\setminus\{\frak{a}\}$. Define an equivalence relation on $\mathcal{O}$ that identifies $x,y\in\mathcal{O}$ if there exists some $\gamma\in\Gamma_\frak{a}$ such that $\gamma x=y$. We then write $[x]=[y]$.

\begin{thm}
For any $\gamma\in\Gamma$ such that $[\![\gamma]\!]$ is non-trivial, $\mathcal{S}_\frak{a}([\gamma.\frak{a}]) = \mathcal{S}_\frak{a}([\![\gamma]\!]).$
\end{thm}

\begin{proof}
Consider the double coset on the LHS of (\ref{dcd}). Let $\bsm a_\gamma&b\\ c_\gamma&d\esm$, $\bsm a' & b'\\ c_\gamma&d_\gamma\esm$ be two other representative of $[\![\gamma]\!]$. Then 
$$
\bpm a_\gamma & b \\ c_\gamma & d\epm^{-1}\bpm a_\gamma & b_\gamma \\ c_\gamma & d_\gamma\epm,\ \bpm a_\gamma & b_\gamma \\ c_\gamma & d_\gamma\epm \bpm  a' & b' \\ c_\gamma & d_\gamma\epm^{-1} = \bpm 1 & * \\ &1\epm \in\left(\sigma_\frak{a}^{-1}\Gamma_\frak{a}\sigma_\frak{a}\right).
$$ 
That is, any double coset representative of $[\![\gamma]\!]$ is completely determined by either its first column or its first row. Therefore, the double cosets $[\![\bsm a_\gamma &*\\ c_\gamma &*\esm]\!] \in (\sigma_\frak{a}^{-1}\Gamma_\frak{a}\sigma_\frak{a})\backslash(\sigma_\frak{a}^{-1}\Gamma\sigma_\frak{a})\slash(\sigma_\frak{a}^{-1}\Gamma_\frak{a}\sigma_\frak{a})$ are in one-to-one correspondence with the cusp points $[a_\gamma/c_\gamma]=a_\gamma/c_\gamma\mod 1$ in  $(\sigma_\frak{a}^{-1}\Gamma_\frak{a}\sigma_\frak{a})\backslash(\sigma_\frak{a}^{-1}\Gamma\sigma_\frak{a}).\infty$.
\end{proof}

\subsection{Selected properties}

\begin{lm}\label{sign}
 For any non-trivial double coset $[\![\gamma]\!]$, $\mathcal{S}_\frak{a}(-[\![\gamma]\!])=\mathcal{S}_\frak{a}([\![\gamma]\!]).$
\end{lm}

\begin{proof}
Following the discussion above, we may assume that $c_\gamma>0$. Using (\ref{AGR}),
$$
\rho_\frak{a}\left(-\bsm a_\gamma&b_\gamma\\ c_\gamma&d_\gamma\esm\right) -\frac14 \sign(-c_\gamma) =\rho_\frak{a}(-I) + \rho_\frak{a}\bsm a_\gamma&b_\gamma\\c_\gamma&d_\gamma\esm +\omega\left(-I,\bsm a_\gamma&b_\gamma\\c_\gamma&d_\gamma\esm\right) +\frac14. 
$$
With (\ref{values}), we can check that 
\begin{align*}
\rho_\frak{a}(I) &=-\omega(I,I)=0,\\
2\rho_\frak{a}(-I) &=\rho_\frak{a}(I)-\omega(-I,-I)=-1,\\ 
\text{and } & \omega\left(-I,\bsm a&b\\c&d\esm\right)=0, 
\end{align*}
The statement follows.
\end{proof}

\begin{lm}
For any non-trivial double coset $[\![\gamma]\!]$,  $\mathcal{S}_\frak{a}([\![\gamma]\!]) = -\mathcal{S}_\frak{a}([\![\gamma^{-1}]\!])$.
\end{lm}

\begin{proof}
By (\ref{AGR}), $\rho_\frak{a}(g) + \rho_\frak{a}(g^{-1}) =\omega(g,g^{-1})$ and (\ref{values}) yields $\omega(g,g^{-1})=0$. The statement follows.
\end{proof}

\subsection{Reciprocity of Dedekind symbols}
\begin{proof}[Proof of \cref{THM1}]
For $c_\gamma c_\tau c_{\gamma\tau}\neq0$,
\begin{align*}
\frac{a_\gamma + d_\gamma}{c_\gamma} + \frac{a_\tau + d_\tau}{c_\tau} - \frac{a_{\gamma\tau} + d_{\gamma\tau}}{c_{\gamma\tau}} &= \frac{a_\gamma + d_\gamma}{c_\gamma} + \frac{a_\tau + d_\tau}{c_\tau} - \frac{(a_\gamma a_\tau +b_\gamma c_\tau) + (c_\gamma b_\tau +d_\gamma d_\tau)}{c_\gamma a_\tau + d_\gamma c_\tau}\\
&= \frac{c_\gamma^2 + c_\tau^2 + (c_\gamma a_\tau + d_\gamma c_\tau)^2}{c_\gamma c_\tau (c_\gamma a_\tau + d_\gamma c_\tau)} = \frac{c_\gamma^2 + c_\tau^2 + c_{\gamma\tau}^2}{c_\gamma c_\tau c_{\gamma\tau}}
\end{align*}
and thus, by (\ref{AGR}),
\begin{align*}
\mathcal{S}_\frak{a}([\![\gamma]\!])+\mathcal{S}_\frak{a}([\![\tau]\!]) - \mathcal{S}_\frak{a}([\![\gamma\tau]\!]) &=
 \frac{V}{4\pi} \left( \frac{c_\gamma}{  c_{\gamma\tau}c_\tau} + \frac{c_\tau}{c_\gamma c_{\gamma\tau}} + \frac{c_{\gamma\tau}}{c_\tau c_\gamma }\right) + \omega(\gamma,\tau) + \\
 &\qquad \qquad +\frac14\left(\sign(c_{\gamma\tau}) -\sign(c_\gamma) -\sign(c_\tau)\right)\\
 &=  \frac{V}{4\pi} \left( \frac{c_\gamma}{  c_{\gamma\tau}c_\tau} + \frac{c_\tau}{c_\gamma c_{\gamma\tau}} + \frac{c_{\gamma\tau}}{c_\tau c_\gamma }\right)  -\frac14 \sign(c_\gamma c_\tau c_{\gamma\tau})
\end{align*}
by \cref{omega values}.
\end{proof}

\section{Explicit formulas for Hecke triangle groups}

\subsection{Parametrizations}
Fix the scaling
$
\sigma_\infty = \bsm \sqrt{\lambda_q} & \\& 1/\sqrt{\lambda_q} \esm
$
and set 
\begin{align*}
 s & :=\sigma_\infty^{-1} S\sigma_\infty = \bpm & -1/\lambda_q\\ \lambda_q&\epm,\\
  u & := \sigma_\infty^{-1}T^n_q \sigma_\infty = \bpm 1&n\\&1\epm,\\
v & := \sigma_\infty^{-1}S^n_q\sigma_\infty = \bpm 1&\\ n\lambda_q^2&1\epm,\\
  \gamma_q & := \sigma_\infty^{-1} \gamma\sigma_\infty = \bpm a & b/\lambda_q \\ c\lambda_q & d\epm,
\end{align*}
for any $n\in\Z$ and any $\gamma=\bsm a&b\\c&d\esm\in G_q$. We record the following formulas for later computations. Let $V_q=\vol(G_q\backslash\h)$. Using the Gauss--Bonnet formula, $V_q=\pi(1-2/q)$.

\begin{lm}
For each $n\in\Z$,
$$
\rho(u^n) = \frac{V_q n}{4\pi}.
$$
\end{lm}
\begin{proof}
This is Lemma 4.2 in \cite{moi}.
\end{proof}

\begin{lm}\label{useful}
Let $g=\bsm a&b\\ c&d\esm\in\SL_2(\R)$. Then
\begin{align*}
\omega(g,s)&= \begin{dcases} 1 & \text{ if }\ c>0\text{ and } d<0,\\ 0 & \text{ otherwise;} \end{dcases} \\
\omega(g,u^n) &= 0\ \text{ for all } n\in\Z \\
\rho(s) &= -1/4\\
\rho(v^n) &= - \rho(u^n)\ \text{ for all } n\in\Z.
\end{align*}
\end{lm}
\begin{proof}
Follows from \cref{corol}.
\end{proof}

\subsection{Dedekind's reciprocity law for Hecke triangle groups}

\begin{proof}[Proof of \cref{thm HTG}]
Let $\gamma=\bsm a&*\\ c&*\esm\in G_q$. We note that
$$
\mathcal{S}\left(\left[\frac{a}{c}\right]\right) -\mathcal{S}\left(\left[\frac{c}{a}\right]\right) = \mathcal{S}([\![\gamma]\!]) -\mathcal{S}([\![S\gamma]\!]) = \mathcal{S}([\![\gamma]\!]) +\mathcal{S}([\![\gamma^{-1}S]\!]) +\mathcal{S}([\![S]\!]).
$$
The reciprocity law now follows as a corollary to \cref{THM1}, or can be seen directly from
\begin{align*}
\mathcal{S}([\![\gamma]\!]) - \mathcal{S}([\![S\gamma]\!]) & = \frac{V_q}{4\pi\lambda_q}\left(\frac{c}{a}+\frac{1}{ac} + \frac{c}{a}\right) + \rho(s) +\omega(s,\gamma_q)+\frac14\left(\sign(d) -\sign(c)\right) \\
& = \frac{1-\frac{2}{q}}{8\cos(\frac{\pi}{q})}\left(\frac{c}{a}+\frac{1}{ac} + \frac{c}{a}\right) +\rho(s) +\frac{1}{4} -\frac14 \sign(ac)\\
& = \frac{1-\frac{2}{q}}{8\cos(\frac{\pi}{q})}\left(\frac{c}{a}+\frac{1}{ac} + \frac{c}{a}\right) -\frac14 \sign(ac)
\end{align*}
by applying successively \cref{omega values} and \cref{useful}.
\end{proof}

\subsection{A formula for the Dedekind symbol attached to a Hecke triangle group}

\begin{proof}[Proof of \cref{THM3}]
We proceed by induction on the length $n$ of words of the form
$$
S_q^{a_1} T_q^{a_2} \cdots T^{a_{n-1}}_q S_q^{a_n}.
$$ If $n=1$, then, by definition,
$$
\mathcal{S}([\![ S_q^{a_1}]\!]) = \frac{V_q}{4\pi} \frac{a+d}{c\lambda_q} - \rho(v^{a_1}) - \frac14  = \frac{V_q}{4\pi} \frac{a+d}{c\lambda_q}+ a_1 - \frac14, 
$$
and, we can check with \cref{useful} that
$$
\mathcal{S}([\![ S_q^{a_1} T_q^{a_2} S]\!]) = \frac{V_q}{4\pi} \frac{a+d}{c\lambda_q} - \rho(v^{a_1}u^{a_2}s) -\frac14 =  \frac{V_q}{4\pi} \frac{a+d}{c\lambda_q}  + a_1 - a_2
$$
for $n=2$. For $n\geq 3$, let
$$
\gamma_n = \begin{dcases} S_q^{a_1}T^{a_2}_q \cdots T^{a_{n-1}}_q S_q^{a_n} & \text{ if } n \text{ is odd,} \\
S_q^{a_1} T^{a_2}_q \cdots T_q^{a_n} S & \text{ if } n \text{ is even,}
\end{dcases}
$$
and observe that 
$$
[\![\gamma_{n-1}]\!] =  \begin{dcases} [\![ \gamma_n ST_q^{a_n}]\!] = [\![ \gamma_n S]\!] & \text{ if } n \text{ is odd,} \\
[\![ \gamma_n S T^{-a_n}_q]\!] = [\![\gamma_n S]\!] & \text{ if } n \text{ is even.}
\end{dcases}
$$
By the induction hypothesis,
$$
\mathcal{S}([\![\gamma_{n-1}]\!]) = \frac{V_q}{4\pi\lambda_q}\left( \gamma_{n-1}.\infty -\gamma_{n-1}^{-1}.\infty - \sum^{n-1}_{j=1} (-1)^j a_j \lambda_q\right) - \frac{1-(-1)^{n-1}}{8},
$$
where 
$$
\gamma_{n-1}.\infty-\gamma_{n-1}^{-1}.\infty = (\gamma_n S).\infty - (T_q^{\mp a_n} S\gamma_n^{-1}).\infty = (\gamma_n S).\infty \pm a_n \lambda_q - (S\gamma^{-1}_n).\infty
$$
with $+a_n \lambda_q$ if $n$ is odd and $-a_n \lambda_q$ otherwise. Together with the reciprocity law,
\begin{align*}
\mathcal{S}([\![\gamma_n]\!]) &= 
\frac{V_q}{4\pi\lambda_q}\left( \frac{d}{c} + \frac{1}{cd} + S\gamma_n^{-1}.\infty\right) - \frac14\sign(cd) + \mathcal{S}([\![\gamma_{n-1}]\!]) \\
& = \frac{V_q}{4\pi\lambda_q}\left( \frac{d}{c}+\frac{1}{cd} +\gamma_n S.\infty -\lambda_q\sum_{j=1}^n (-1)^j a_j\right) -\frac{1-(-1)^n}{8}\\
& = \frac{V_q}{4\pi\lambda_q} \left( \frac{d}{c} +\frac{1}{cd} + \frac{b}{d} - \lambda_q\sum^{n}_{j=1}(-1)^j a_j \right) -  \frac{1-(-1)^{n}}{8}\\
& = \frac{V_q}{4\pi \lambda_q} \left(\frac{a+d}{c} - \lambda_q\sum^{n}_{j=1}(-1)^j a_j \right) -  \frac{1-(-1)^{n-1}}{8}
\end{align*}
and the identification with continued fraction expansions is immediate since $\frac{a}{c}=\gamma_n.\infty$ and $\frac{d}{c}=-\gamma_n^{-1}.\infty.$

\end{proof}

\end{document}